\documentclass[11pt]{amsart}

\usepackage{amssymb,booktabs,caption,cite,enumitem,mathtools}
\usepackage[colorlinks=true,allcolors=blue]{hyperref}
\usepackage[margin=1.5in]{geometry}

\makeatletter
\newlength{\@thlabel@width}%
\newcommand{\thmenumhspace}{\settowidth{\@thlabel@width}{\itshape1.}\sbox{\@labels}{\unhbox\@labels\hspace{\dimexpr-\leftmargin+\labelindent/5+\labelsep+\@thlabel@width-\itemindent}}}
\makeatother

\newtheorem{theorem}{Theorem}
\newtheorem{lemma}[theorem]{Lemma}
\newtheorem{proposition}[theorem]{Proposition}

\DeclareMathOperator{\lcm}{lcm}

\title{On Arithmetical Structures on Complete Graphs}

\author{Zachary Harris}
\address[Z.~Harris]{Department of Mathematics, Niagara University, Niagara University, NY 14109, USA}
\email{zharris.edu@gmail.com}

\author{Joel Louwsma}
\address[J.~Louwsma]{Department of Mathematics, Niagara University, Niagara University, NY 14109, USA}
\email{jlouwsma@niagara.edu}

\begin{document}

\begin{abstract}
An arithmetical structure on the complete graph~$K_n$ with $n$~vertices is given by a collection of $n$~positive integers with no common factor each of which divides their sum. We show that, for all positive integers~$c$ less than a certain bound depending on~$n$, there is an arithmetical structure on~$K_n$ with largest value~$c$. We also show that, if each prime factor of~$c$ is greater than $(n+1)^2/4$, there is no arithmetical structure on~$K_n$ with largest value~$c$. We apply these results to study which prime numbers can occur as the largest value of an arithmetical structure on~$K_n$.
\end{abstract}

\maketitle

\section{Introduction}

How can one have a collection of positive integers, with no common factor, each of which divides their sum? For example, $105$, $70$, $15$, $14$, and~$6$ sum to $210$, which is divisible by each of these numbers. Introducing notation, we seek positive integers $r_1,r_2,\dotsc,r_n$ with no common factor such that 
\begin{equation}\label{eq:arith}
r_j\,\,\Big|\,\sum_{i=1}^nr_i\quad\text{for all~$j$.}
\end{equation}
It is well known that finding such~$r_i$ is equivalent to finding positive integer solutions of the Diophantine equation 
\begin{equation}\label{eq:dio}
\frac{1}{x_1}+\frac{1}{x_2}+\dotsb+\frac{1}{x_n}=1.
\end{equation}
Indeed, given $r_1,r_2,\dotsc,r_n$ satisfying~\eqref{eq:arith}, dividing both sides of the equation $r_1+r_2+\dotsb+r_n=\sum_{i=1}^n r_i$ by $\sum_{i=1}^n r_i$ gives a solution to~\eqref{eq:dio}, and, given a solution of~\eqref{eq:dio}, the numbers $\lcm(x_1,x_2,\dotsc,x_n)/x_i$ satisfy~\eqref{eq:arith} and have no common factor.

Our interest in this question stems from an interest in arithmetical structures. An \emph{arithmetical structure} on a finite, connected graph is an assignment of positive integers to the vertices such that:
\begin{enumerate}[label=\textup{(\alph*)},ref=\textup{\alph*}]
\item At each vertex, the integer there is a divisor of the sum of the integers at adjacent vertices (counted with multiplicity if the graph is not simple).
\item The integers used have no nontrivial common factor.
\end{enumerate}
Arithmetical structures were introduced by Lorenzini~\cite{L89} to study intersections of degenerating curves in algebraic geometry. The usual definition, easily seen to be equivalent to the one given here, is formulated in terms of matrices. From that perspective, an arithmetical structure may be regarded as a generalization of the \emph{Laplacian matrix}, which encodes many important properties of a graph. Notions in this direction that have received a significant amount of attention include the sandpile group and the chip-firing game; for details, see \cite{CP,K18,GK}.

On the \emph{complete graph}~$K_n$ with $n$~vertices, positive integers $r_1,r_2,\dotsc,r_n$ with no common factor give an arithmetical structure if and only if
\[
r_j\,\,\Big|\,\sum_{\substack{i=1\\i\neq j}}^nr_i\quad\text{for all~$j$};
\]
it is immediate that this condition is equivalent to~\eqref{eq:arith}. Therefore, in this language, the opening question of this paper seeks arithmetical structures on complete graphs. 

Lorenzini \cite[Lemma~1.6]{L89} shows that there are finitely many arithmetical structures on any finite, connected graph, but his result does not give a bound on the number of structures. Several recent papers \cite{B18,A20,GW} count arithmetical structures on various families of graphs, including path graphs, cycle graphs, bidents, and certain path graphs with doubled edges. However, counting arithmetical structures on complete graphs is a difficult problem. The number of arithmetical structures on~$K_n$ for $n\leq8$ is given in~\cite{OEIS}. For general~$n$, bounds have been obtained by several authors \cite{EG,S03,BE,K14} working from the perspective of the Diophantine equation~\eqref{eq:dio}. Other papers such as \cite{B07,B08,ACF} determine, for specific~$n$, the number of solutions of~\eqref{eq:dio} satisfying certain additional conditions on the~$x_i$.

It is conjectured in \cite[Conjecture~6.10]{CV} that, for any connected, simple graph~$G$ with $n$~vertices, the number of arithmetical structures on~$G$ is at most the number of arithmetical structures on~$K_n$. To approach this conjecture, one would like a better understanding of the types of arithmetical structures that occur on complete graphs. In this direction, this paper studies which positive integers can occur as the largest value of an arithmetical structure on~$K_n$. Clearly the~$r_i$ of an arithmetical structure can be permuted; in the following we make the assumption $r_1\geq r_2\geq \dotsb\geq r_n$. We construct arithmetical structures to show that $r_1$ can take certain values and give obstructions to show that it cannot take other values. 

Our primary construction theorem (Theorem~\ref{thm:const}) shows that $r_1$ can take any value up to a certain bound depending on~$n$. More specifically, $r_1$ can be any positive integer less than or equal to $\max_{k\in\mathbb{Z}_{>0}}(2^kn-(k+2^k-2)2^k-1)$. This bound improves somewhat if we restrict attention to prime numbers; $r_1$ can be any prime number less than or equal to $\max_{k\in\mathbb{Z}_{>0}}(2^kn-(k+2^k-3)2^k-3)$. 

We also prove an obstruction theorem (Theorem~\ref{thm:genobstr}) that shows $r_1$ cannot take any value all of whose prime factors are greater than $(n+1)^2/4$. Restricting attention to prime numbers, this bound improves to show that $r_1$ cannot be any prime number greater than $n^2/4+1$ (Theorem~\ref{thm:obstr}).

The final section focuses on the possible prime values $r_1$ can take. We explicitly check prime numbers in the gap between the bound of Theorem~\ref{thm:const} and the bound of Theorem~\ref{thm:obstr}, showing that $r_1$ can take some of these values but not others. In particular, we observe that there can be prime numbers $p_1$ and~$p_2$ with $p_1<p_2$ such that there is an arithmetical structure on~$K_n$ with $r_1=p_2$ but no arithmetical structure on~$K_n$ with $r_1=p_1$. 

\section{Construction}\label{sec:construction}

In this section, we show how to construct arithmetical structures on complete graphs with certain values of~$r_1$. Our main construction theorem is the following. 

\begin{theorem}\label{thm:const}
\begin{enumerate}[label=\textup{(\alph*)},ref=\textup{\alph*}]\thmenumhspace
\item For any positive integer $c\leq \max_{k\in\mathbb{Z}_{>0}}(2^kn-(k+2^k-2)2^k-1)$, there is an arithmetical structure on~$K_n$ with $r_1=c$.
\item For any prime number $p\leq \max_{k\in\mathbb{Z}_{>0}}(2^kn-(k+2^k-3)2^k-3)$, there is an arithmetical structure on~$K_n$ with $r_1=p$.
\end{enumerate}
\end{theorem}

We establish Propositions \ref{prop:const1}, \ref{prop:const2a}, and~\ref{prop:const2b} on the way to proving Theorem~\ref{thm:const}.

\begin{proposition}\label{prop:const1}
For any positive integer $c\leq n-1$, there is an arithmetical structure on~$K_n$ with $r_1=c$.
\end{proposition}

\begin{proof}
Let
\[
r_i=\begin{cases}
c&\text{for }i\in\{1,2,\dotsc,n-c\},\\
1&\text{for }i\in\{n-c+1,n-c+2,\dotsc,n\}.
\end{cases}
\]
Then
\[
\sum_{i=1}^n r_i=c(n-c)+c=c(n-c+1).
\]
Since this is divisible by both $c$ and~$1$, we have thus produced an arithmetical structure on~$K_n$.
\end{proof}

Before turning to Propositions \ref{prop:const2a} and~\ref{prop:const2b}, we establish the following lemma.

\begin{lemma}\label{lem:binary}
\begin{enumerate}[label=\textup{(\alph*)},ref=\textup{\alph*}]\thmenumhspace
\item Let $k\in\mathbb{Z}_{\geq0}$ and $\ell\in\mathbb{Z}_{>0}$ with $k\leq\ell$. Every integer~$c$ satisfying $\ell\leq c\leq(\ell-k+1)2^k-1$ can be expressed as $\sum_{j=1}^\ell2^{k_j}$ for some $k_j\in\{0,1,\dotsc,k\}$, where $k_j=0$ for some $j\in\{1,2,\dotsc,\ell\}$. 
\item Let $k\in\mathbb{Z}_{\geq0}$ and $\ell\in\mathbb{Z}_{>0}$ with $k\leq\ell$. Every odd integer~$c$ satisfying $\ell\leq c\leq(\ell-k+2)2^k-3$ can be expressed as $\sum_{j=1}^\ell2^{k_j}$ for some $k_j\in\{0,1,\dotsc,k\}$, where $k_j=0$ for some $j\in\{1,2,\dotsc,\ell\}$. 
\end{enumerate}
\end{lemma}

\begin{proof}
To show~(a), we proceed by induction on~$c$. In the base case, $c=\ell$, we can let $k_j=0$ for all~$j$ and have $c=\sum_{j=1}^{\ell}2^{k_j}$. Now suppose $c=\sum_{j=1}^{\ell}2^{k_j}$ for $c\leq (\ell-k+1)2^k-2$. Then $k_j=k$ for at most $\ell-k$ values of~$j$, meaning $k_j<k$ for at least $k$~values of~$j$. If among these values we had each of $0,1,\dotsc,k-1$ only once, we would then have $\sum_{j=1}^{\ell}2^{k_j}=(\ell-k+1)2^k-1>(\ell-k+1)2^k-2$. Therefore there is some $b<k$ for which $k_{j_1}=b=k_{j_2}$ for some $j_1\neq j_2$. Define
\[
k_j'=\begin{cases}
k_j+1&\text{for }j=j_1,\\
0&\text{for }j=j_2,\\
k_j&\text{otherwise.}
\end{cases}
\]
Then $\sum_{j=1}^{\ell}2^{k_j'}=\sum_{j=1}^{\ell}2^{k_j}+1=c+1$. The result follows.

For~(b), first note that if $c\leq (\ell-k+1)2^k-1$ then the result follows from~(a). Therefore, assume $c>(\ell-k+1)2^k-1$ and let $c'=c-((\ell-k+1)2^k-1)$, noting that $c'\leq 2^k-2$. Since $c$ is odd, $c'$ must be even. Therefore $c'$ can be written in the form $\sum_{j=1}^{k-1} s_j2^j$, where each~$s_j$ is either $0$ or~$1$; the~$s_j$ are iteratively determined in reverse by letting $s_j=1$ if $c'-\sum_{i=j+1}^{k-1}s_i2^i\geq2^j$ and letting $s_j=0$ otherwise. Define
\[
k_j=\begin{cases}
0&\text{for }j=1,\\
j-1+s_{j-1}&\text{for }j\in\{2,3,\dotsc,k\},\\
k&\text{for }j\in\{k+1,k+2,\dotsc,\ell\}.
\end{cases}
\]
Then
\[
\sum_{j=1}^{\ell}2^{k_j}=\sum_{j=1}^k 2^{j-1}+\sum_{j=2}^ks_{j-1}2^{j-1}+\sum_{j=k+1}^{\ell}2^k=2^k-1+c'+(\ell-k)2^k=c.\qedhere
\]
\end{proof}

We use Lemma~\ref{lem:binary} to prove Propositions \ref{prop:const2a} and~\ref{prop:const2b}.

\begin{proposition}\label{prop:const2a}
Fix $n\geq 2$. For any positive integer~$k$ satisfying $k+2^k-1\leq n$ and any positive integer~$c$ satisfying $n-2^k+1\leq c\leq (n-k-2^k+2)2^k-1$, there is an arithmetical structure on~$K_n$ with $r_1=c$.\end{proposition}

\begin{proof}
Let $r_i=c$ for all $i\in\{1,2,\dotsc,2^k-1\}$. Let $\ell=n-2^k+1$, noting that our assumptions guarantee that $k\leq\ell$ and $\ell\leq c\leq (\ell-k+1)2^k-1$. Lemma~\ref{lem:binary}(a) then shows how to write $c=\sum_{j=1}^{\ell}2^{k_j}$. We use the values $2^{k_j}$, in decreasing order, to define $r_i$ for $i\in\{2^k,2^k+1,\dotsc,n\}$, noting that $r_n=1$. Then
\[
\sum_{i=1}^nr_i=(2^k-1)c+c=2^kc.
\]
Since $2^kc$ is divisible by~$c$ and by $2^{k'}$ for all $k'\in\{0,1,\dotsc,k\}$, we have thus produced an arithmetical structure on~$K_n$.
\end{proof}

Although we imposed the condition $k+2^k-1\leq n$ here to ensure that $k\leq\ell$ in the proof, this does not restrict possible values of~$r_1$, as we show in the proof of Theorem~\ref{thm:const}. Together with Proposition~\ref{prop:const1}, Proposition~\ref{prop:const2a} with $k=1$ allows us to construct arithmetical structures on~$K_n$ with $r_1$ taking any value up to $2n-3$. When $k=2$, the bound is $4n-17$; when $k=3$, the bound is $8n-73$; and when $k=4$, the bound is $16n-289$. If we restrict attention to prime~$r_1$, these bounds can be improved slightly, as the following proposition shows.

\begin{proposition}\label{prop:const2b}
Fix $n\geq2$. For any positive integer~$k$ satisfying $k+2^k-1\leq n$ and any prime number~$p$ satisfying $n-2^k+1\leq p\leq (n-k-2^k+3)2^k-3$, there is an arithmetical structure on~$K_n$ with $r_1=p$.
\end{proposition}

\begin{proof}
If $p=2$ (and $n\geq3$), Proposition~\ref{prop:const1} gives an arithmetical structure on~$K_n$ with $r_1=p$. Therefore suppose $p$ is odd. Let $r_i=p$ for all $i\in\{1,2,\dotsc,2^k-1\}$. Let $\ell=n-2^k+1$, noting that our assumptions guarantee $k\leq \ell$ and $\ell\leq p\leq (\ell-k+2)2^k-3$. Lemma~\ref{lem:binary}(b) then shows how to write $p=\sum_{j=1}^{\ell}2^{k_j}$. As in the proof of Proposition~\ref{prop:const2a}, we use the values $2^{k_j}$, in decreasing order, to define $r_i$ for $i\in\{2^k,2^k+1,\dotsc,n\}$, noting that $r_n=1$. Then
\[
\sum_{i=1}^nr_i=(2^k-1)p+p=2^kp,
\]
which is divisible by~$p$ and by~$2^{k'}$ for all $k'\in\{0,1,\dotsc,k\}$, so therefore we have produced an arithmetical structure on~$K_n$.
\end{proof}

For example, when $k=1$, Proposition~\ref{prop:const2b} allows us to construct arithmetical structures on~$K_n$ with $r_1$ taking prime values as large as $2n-3$. When $k=2$, the bound is $4n-15$; when $k=3$, the bound is $8n-67$; and when $k=4$, the bound is $16n-275$. 

We are now prepared to complete the proof of Theorem~\ref{thm:const}.

\begin{proof}[Proof of Theorem~\ref{thm:const}]
The necessary constructions are given in Propositions \ref{prop:const1}, \ref{prop:const2a}, and~\ref{prop:const2b}. It remains only to show that, for each~$n$, values of~$k$ that maximize the upper bounds in Propositions \ref{prop:const2a} and~\ref{prop:const2b} satisfy $k+2^k-1\leq n$.

The upper bound $(n-k-2^k+2)2^k-1$ in Proposition~\ref{prop:const2a} is linear in~$n$ with slope~$2^k$. A straightforward calculation shows that the bound with slope $2^{k-1}$ coincides with the bound with slope~$2^k$ when $n=k+3\cdot2^{k-1}-1$ and that the bound with slope~$2^k$ coincides with the bound with slope $2^{k+1}$ when $n=k+3\cdot2^k$. Therefore the bound with slope~$2^k$ is maximal exactly when $n$ is between $k+3\cdot2^{k-1}-1$ and $k+3\cdot2^k$. When the bound is maximized, we therefore have that $n\geq k+3\cdot2^{k-1}-1\geq k+2^k-1$, meaning the condition of Proposition~\ref{prop:const2a} is satisfied. This proves~(a).

The argument for~(b) is very similar. The upper bound $(n-k-2^k+3)2^k-3$ in Proposition~\ref{prop:const2b} is maximal for $n$ between $k+3\cdot2^{k-1}-2$ and $k+3\cdot2^k-1$. When the bound is maximized, we then have that $n\geq k+3\cdot2^{k-1}-2\geq k+2^k-1$, meaning the condition of Proposition~\ref{prop:const2b} is satisfied. 
\end{proof}

We conclude this section by giving another construction that allows us to produce some arithmetical structures with values of~$r_1$ other than those guaranteed by Theorem~\ref{thm:const}. 

\begin{proposition}\label{prop:const3}
For any positive integer $k\leq n-1$, there is an arithmetical structure on~$K_n$ with $r_1=k(n-k)+1$. 
\end{proposition}

\begin{proof}
Let
\[
r_i=\begin{cases}
k(n-k)+1&\text{for }i\in\{1,2,\dotsc,k-1\},\\
k&\text{for }i\in\{k,k+1,\dotsc,n-1\},\\
1&\text{for }i=n.
\end{cases}
\]
Then
\[
\sum_{i=1}^n r_i=(k-1)(k(n-k)+1)+k(n-k)+1=k(k(n-k)+1).
\]
Since this is divisible by $k(n-k)+1$, $k$, and~$1$, we have thus produced an arithmetical structure on~$K_n$.
\end{proof}

For example, when $n=13$, Theorem~\ref{thm:const} guarantees that we can find an arithmetical structure on~$K_n$ with $r_1=p$ for all prime $p\leq 37$. By taking $k=5$ in Proposition~\ref{prop:const3}, we can also produce an arithmetical structure with $r_1=41$. By taking $k=6$, we can produce an arithmetical structure with $r_1=43$. The results of this section cannot be extended too much further, as we show in the following section.

\section{Obstruction}\label{sec:obstruction}

We next prove obstruction results that complement our constructions in the previous section. Our first result shows that $r_1$ cannot be a product of primes all of which are too large.

\begin{theorem}\label{thm:genobstr}
Suppose $c\geq2$ is an integer with prime factorization $p_1^{a_1}p_2^{a_2}\dotsm p_k^{a_k}$, where $p_1<p_2<\dotsb<p_k$ and $a_i\geq1$ for all~$i$. If $p_1>(n+1)^2/4$, then there is no arithmetical structure on~$K_n$ with $r_1=c$. 
\end{theorem}

\begin{proof}
Suppose we have an arithmetical structure on~$K_n$ with $r_1=c$. Knowing that $r_1\mid\sum_{i=1}^nr_i$, we define $b=\sum_{i=1}^nr_i/r_1$. Then $\sum_{i=1}^nr_i=bc$, meaning that $r_i\mid bp_1^{a_1}p_2^{a_2}\dotsm p_k^{a_k}$ for all~$i$. Let $m$ be the largest value of~$i$ for which $r_i=c$. For all $i\in\{m+1,m+2\dotsc,n\}$, we have $r_i<c$, which implies that $r_i\leq bp_1^{a_1-1}p_2^{a_2}\dotsm p_k^{a_k}$. This means $\sum_{i=m+1}^nr_i\leq (n-m)bp_1^{a_1-1}p_2^{a_2}\dotsm p_k^{a_k}$. We also have that $\sum_{i=m+1}^nr_i=(b-m)p_1^{a_1}p_2^{a_2}\dotsm p_k^{a_k}$. Therefore
\[
(b-m)p_1^{a_1}p_2^{a_2}\dotsm p_k^{a_k}\leq (n-m)bp_1^{a_1-1}p_2^{a_2}\dotsm p_k^{a_k},
\]
and hence $(b-m)p_1\leq (n-m)b$. When $b=m$, there is only one arithmetical structure on~$K_n$, namely that with $r_i=1$ for all~$i$, so the desired structure cannot arise in this case. Therefore we assume $b>m$, in which case we have
\[
p_1\leq\frac{(n-m)b}{b-m}.
\]

When $b=m+1$, this gives $p_1\leq(n-b+1)b$. It is a simple calculus exercise to show this bound is maximized when $b=(n+1)/2$. It follows that $p_1\leq (n+1)^2/4$.

When $b\geq m+2$, we have that
\[
p_1\leq \frac{(n-m)b}{b-m}=\frac{nb-mb+nm-nm}{b-m}=n+\frac{m(n-b)}{b-m}\leq n+\frac{m(n-m-2)}{2}.
\]
It is a simple calculus exercise to show this bound is maximized when $m=n/2-1$, so therefore 
\[
p_1\leq n+\frac{(n/2-1)(n/2-1)}{2}=\frac{n^2}{8}+\frac{n}{2}+\frac{1}{2}=\frac{(n+1)^2}{4}-\frac{n^2-1}{4}\leq\frac{(n+1)^2}{4}.
\]

The result follows.
\end{proof}

If we restrict attention to arithmetical structures where $r_1$ is a prime number, then Theorem~\ref{thm:genobstr} can be improved to Theorem~\ref{thm:obstr}. The general outline of the proof is similar, with some of the bounds improved.

\begin{theorem}\label{thm:obstr}
If $p$ is a prime number with $p>n^2/4+1$, then there is no arithmetical structure on~$K_n$ with $r_1=p$.
\end{theorem}

\begin{proof}
If $p=2$, the hypothesis of the theorem is only satisfied for $n=1$, and there is no arithmetical structure on~$K_1$ with $r_1=2$. Suppose we have an arithmetical structure on~$K_n$ with $r_1=p$, where $p\geq3$. Knowing that $r_1\mid\sum_{i=1}^nr_i$, we define $b=\sum_{i=1}^nr_i/r_1$. Then $\sum_{i=1}^nr_i=bp$, meaning that $r_i\mid bp$ for all~$i$. Let $m$ be the largest value of~$i$ for which $r_i=p$. We can only have $b=m$ if $r_i=1$ for all~$i$, but then $r_1$ is not prime. We consider two cases: when $b=m+1$ and when $b\geq m+2$.

Case I: $b=m+1$. For all $i\in\{m+1,m+2,\dotsc,n\}$, we have that $r_i\mid bp$ and $r_i<p$, so therefore $r_i\mid b$. Whenever $r_i<b$, this means $r_i\leq b/2$. If $r_{n-1},r_n<b$, we would then have that $\sum_{i=m+1}^nr_i\leq (n-m-1)b$. If instead $r_i=b$ for all $i\in\{m+1,m+2,\dotsc,n-1\}$, we would have that $r_n\mid r_i$ for all $i\in\{m+1,m+2,\dotsc,n\}$. Since $\sum_{i=m+1}^n r_i=(b-m)p=p$, this would also mean $r_n\mid p$, and hence that $r_n\mid r_i$ for all~$i$. Therefore we would need to have $r_n=1$, meaning that $\sum_{i=m+1}^nr_i\leq (n-m-1)b+1$. Regardless of the value of $r_{n-1}$, we thus have that 
\[
p=\sum_{i=m+1}^nr_i\leq (n-m-1)b+1=(n-b)b+1=nb-b^2+1.
\]
It is a simple calculus exercise to show that this bound is maximized when $b=n/2$. Hence we have that $p\leq n^2/4+1$.

Case II: $b\geq m+2$. 
We have that $r_i\leq b$ for all $i\in\{m+1,m+2,\dotsc,n\}$ and $\sum_{i=m+1}^nr_i=(b-m)p$, so therefore $(b-m)p\leq (n-m)b$. As in the proof of Theorem~\ref{thm:genobstr}, this yields that 
\[
p\leq \frac{(n-m)b}{b-m}\leq n+\frac{m(n-m-2)}{2}, 
\]
and this bound is maximized when $m=n/2-1$. Therefore
\[
p\leq n+\frac{(n/2-1)(n/2-1)}{2}=\frac{n^2}{8}+\frac{n}{2}+\frac{1}{2}
=\frac{n^2}{4}+1-\frac{(n-2)^2}{8}<\frac{n^2}{4}+1.
\]

We have thus shown that in all cases we must have $p\leq n^2/4+1$. 
\end{proof}

For even~$n$, we can choose $k=n/2$ in Proposition~\ref{prop:const3} and get an arithmetical structure on~$K_n$ with $r_1=n^2/4+1$. For odd~$n$, we can choose $k=(n-1)/2$ and get an arithmetical structure on~$K_n$ with $r_1=(n^2-1)/4+1$. As some of these values of~$r_1$ are prime, the bound in Theorem~\ref{thm:obstr} therefore cannot be improved.

There are arithmetical structures for which $r_1$ takes composite values larger than the bound given in Theorem~\ref{thm:obstr}. For instance, the example in the opening paragraph of this paper gives an arithmetical structure on~$K_5$ with $r_1=105$. 

\section{Prime \texorpdfstring{$r_1$}{r\_1}}

This section considers the possible prime values $r_1$ can take in an arithmetical structure on~$K_n$. Theorem~\ref{thm:const}(b) guarantees that $r_1$ can take any prime value up to $2^kn-(k+2^k-3)2^k-3$ for any~$k$. Theorem~\ref{thm:obstr} says that $r_1$ cannot take any prime value larger than $n^2/4+1$. These bounds are not too far from each other. The function $n^2/4+1$ has linear approximations of the form $2^kn-2^{2k}+1$. When $k$ is $1$ or~$2$, this linear approximation coincides with the bound from Theorem~\ref{thm:const}(b). In general, it differs from this bound by $(k-3)2^k+4$.

Proposition~\ref{prop:const3} shows that $r_1$ can take some of the prime values in the gap between the bound of Theorem~\ref{thm:const}(b) and the bound of Theorem~\ref{thm:obstr}. We can check by hand whether it can take other prime values; to illustrate how to do this, we explain why there is no arithmetical structure on~$K_{18}$ with $r_1=79$. Suppose there were such a structure, and let $b=\sum_{i=1}^{18}r_i/r_1$, so that $\sum_{i=1}^{18}r_i=79b$. Let $m$ be the largest value of~$i$ for which $r_i=79$. Then $\sum_{i=m+1}^{18}r_i=79(b-m)$. For all $i\in\{m+1,m+2,\dotsc,18\}$, we have that $r_i\mid 79b$ and $r_i<79$. Therefore $r_i\mid b$, and hence $r_i\leq b$. This means that $\sum_{i=m+1}^{18}r_i\leq(18-m)b$, so we must have $79(b-m)\leq (18-m)b$. If $b\geq m+2$, we would have that 
\[
61b-79m+mb\geq61(m+2)-79m+m(m+2)=(m-8)^2+58>0,
\]
which would imply that $79(b-m)>(18-m)b$. Therefore we cannot have $b\geq m+2$. Since $b=m$ is only possible if $r_1=1$, it therefore remains to consider whether we can have $b=m+1$. In this case, we would have $\sum_{i=m+1}^{18}r_i\leq (18-m)(m+1)$. This bound is less than~$79$ except when $m$ satisfies $6\leq m\leq 11$. If $m=6$, we would need to have $12$~divisors of~$7$ that sum to~$79$, but this is not possible. If $m=7$, we would need to have $11$~divisors of~$8$ that sum to~$79$, but this is not possible. If $m=8$, we would need to have $10$~divisors of~$9$ that sum to~$79$, but this is not possible. If $m=9$, we would need to have $9$~divisors of~$10$ that sum to~$79$, but this is not possible. If $m=10$, we would need to have $8$~divisors of~$11$ that sum to~$79$, but this is not possible. If $m=11$, we would need to have $7$~divisors of~$12$ that sum to~$79$, but this is not possible. Therefore there is no arithmetical structure on~$K_{18}$ with $r_1=79$. A similar approach can be used to either find arithmetical structures with other prime values of~$r_1$ or to show that they do not exist. We have done this for all $n\leq27$; the results are shown in Table~\ref{tab:prime}.

\begin{table}
\begin{tabular}{cccccc}
\toprule
$n$&Yes, Thm.~\ref{thm:const}(b)&No, Thm.~\ref{thm:obstr}&Yes, Prop.~\ref{prop:const3}&Yes, other&No, other\\
\midrule
3&$p\leq3$&$p>3.25$&&&\\
4&$p\leq5$&$p>5$&&&\\
5&$p\leq7$&$p>7.25$&&&\\
6&$p\leq9$&$p>10$&&&\\
7&$p\leq13$&$p>13.25$&&&\\
8&$p\leq17$&$p>17$&&&\\
9&$p\leq21$&$p>21.25$&&&\\
10&$p\leq25$&$p>26$&&&\\
11&$p\leq29$&$p>31.25$&$31$&&\\
12&$p\leq33$&$p>37$&$37$&&\\
13&$p\leq37$&$p>43.25$&$41,43$&&\\
14&$p\leq45$&$p>50$&&$47$&\\
15&$p\leq53$&$p>57.25$&&&\\
16&$p\leq61$&$p>65$&&&\\
17&$p\leq69$&$p>73.25$&$71,73$&&\\
18&$p\leq77$&$p>82$&&&$79$\\
19&$p\leq85$&$p>91.25$&$89$&&\\
20&$p\leq93$&$p>101$&$97,101$&&\\
21&$p\leq101$&$p>111.25$&$109$&$103,107$&\\
22&$p\leq109$&$p>122$&$113$&&\\
23&$p\leq117$&$p>133.25$&$127,131$&&\\
24&$p\leq125$&$p>145$&&$127,131,137,139$&\\
25&$p\leq133$&$p>157.25$&$137,151,157$&$139,149$&\\
26&$p\leq141$&$p>170$&&$149,151,157,163$&$167$\\
27&$p\leq149$&$p>183.25$&$163,181$&$151,157,167,173$&$179$\\
\bottomrule
\end{tabular}
\caption{Possible prime~$r_1$ in arithmetical structures on~$K_n$ for $n\leq 27$.\label{tab:prime}}
\end{table}

We conclude by noting that, on~$K_{27}$, there is no arithmetical structure with $r_1=179$ whereas there is an arithmetical structure with $r_1=181$. This shows that there is not a cutoff function $f(n)$ such that, for each~$n$, there is an arithmetical structure on~$K_n$ with $r_1=p$ for all primes $p\leq f(n)$ and no such structure for any $p>f(n)$. Therefore, while one could attempt to improve the bound of Theorem~\ref{thm:const}(b), the possible prime values of~$r_1$ cannot be fully explained by a result of this form.

\section*{Acknowledgments}
We would like to thank Nathan Kaplan for a helpful conversation and Darleen Perez-Lavin for helpful comments on a previous version of this paper. The first author was supported by a Niagara University Undergraduate Student Summer Support Grant.

\bibliographystyle{amsplain}
\bibliography{CompleteGraphs}

\end{document}